\documentclass[11pt]{amsart}
\usepackage{amsmath, amsthm, amssymb,amscd}
\usepackage{float}
\usepackage{graphicx}
\usepackage{xypic}
\usepackage[arrow, matrix, curve]{xy}
\usepackage{color}
\usepackage{enumitem}  
\usepackage{tikz-cd} 
\usepackage{url}

\usepackage{blindtext}

\usepackage[english]{babel}
\newcommand{\seq}{\subseteq}
\newcommand{\C}{\mathbb{C}}

\newcommand{\rank}{\rm{rank }}
\newtheorem{thm}{Theorem}[section]
\newtheorem*{thm-nl}{Theorem}
\newtheorem*{prop-nl}{Proposition}

\newtheorem{lem}[thm]{Lemma}

\def\PP{{\mathbf P}}

\def\Pic0{{\rm Pic}^0(X)}

\newtheorem{cor}[thm]{Corollary}
\newtheorem*{cor-nl}{Corollary}
\newtheorem{conjecture}[thm]{Conjecture}
\newtheorem*{conjecture-nl}{Conjecture}

\newtheorem*{quest-nl}{Question}
\newtheorem*{quests-nl}{Questions}

\newtheorem{prop}[thm]{Proposition}

\theoremstyle{remark}

\title{{The Geometric Syzygy Conjecture in Even Genus}}
\date{\today}

\author[M. Kemeny]{Michael Kemeny}

\address{University of Wisconsin-Madison, Department of Mathematics, 480 Lincoln Dr
\hfill \newline\texttt{}
 \indent WI 53706, USA} \email{{\tt michael.kemeny@gmail.com}}

\begin{document}
\begin{abstract}
We prove the Geometric Syzygy Conjecture for generic canonical curves of even genus. This result extends Green's classical result \cite{green-quadrics} on the generation of the ideal of a canonical curve by rank four quadrics to the highest linear syzygy group.

\end{abstract}
\maketitle
\setcounter{section}{-1}
\section{Introduction}
Let $X$ be a projective variety with very ample line bundle $L$. The study of the equations of the embedded variety $X \hookrightarrow \PP(H^0(L)^{\vee})$ is governed by the \emph{Koszul cohomology spaces} $K_{p,q}(X,L)$, \cite{green-koszul}, defined as the middle cohomology of the sequence
$$\bigwedge^{p+1}H^0(L) \otimes H^0(L^{\otimes q-1}) \to \bigwedge^{p}H^0(L) \otimes H^0(L^{\otimes q}) \to \bigwedge^{p-1}H^0(L) \otimes H^0(L^{\otimes q+1}).$$
\medskip 

The first question one asks about the vector spaces $K_{p,q}(X,L)$ is, of course, to determine their dimension. Once the dimension is known we can try to find a basis with special properties. In the case of the \emph{linear} syzygy spaces $K_{p,1}(X,L)$, there is a well-defined notion of rank.  A natural question, with roots going back at least to Andreotti--Mayer's 1967 paper \cite{andreotti-mayer}, is whether one can find a spanning set of $K_{p,1}(X,L)$ consisting of elements of low rank. \medskip

To be more precise, the rank of a linear syzygy $\alpha \in K_{p,1}(X,L)$ is the dimension of the minimal linear subspace
$V \seq H^0(X,L)$ such that $\alpha \in K_{p,1}(X,L,V)$, i.e.\ such that $\alpha$ is represented by an element of $\bigwedge^p V \otimes H^0(L) \seq \bigwedge^p H^0(L) \otimes H^0(L).$

  \medskip

 Syzygies of low rank have geometric meaning. If $X$ is nondegenerate we have $\rm{rank}(\alpha) \geq p+1$. If there exists a syzygy $\alpha$ with $\rm{rank}(\alpha)=p+1$, then $X$ lies on a rational normal scroll. More precisely, the \emph{syzygy scheme} $$\rm{Syz}(\alpha) \seq \PP(H^0(L)^{\vee})$$ of $\alpha$, as defined by Green \cite{green-canonical}, defines a scroll. Similarly, syzygies of rank $p+2$ arise from linear sections of Grassmannians. More precisely, $\rm{Syz}(\alpha)$ contains the cone over a linear section of $\rm{Gr}_2(V^{\vee} \oplus \C)$. For proofs of these facts see \cite{bothmer-JPAA}, \cite{aprodu-nagel-nonvanishing}. \medskip
 
 Correspondingly, a syzygy $\alpha \in K_{p,1}(X,L)$ is said to be \emph{geometric} if $\rank(\alpha) \leq p+2$. Most general constructions of syzygies produce geometric syzygies. In particular, an influential construction of Green--Lazarsfeld \cite{green-koszul} produces geometric syzygies out of a decomposition $L=L_1 \otimes L_2$ with $h^0(L_i) \geq 2$ for $i=1,2$, \cite[\S 3.4.2]{aprodu-nagel}. Voisin further constructs geometric syzygies out of rank two vector bundles in her proof of Green's conjecture for curves of odd genus, \cite{V2}, \cite{aprodu-nagel-nonvanishing}.\medskip
 
 %Moreover, Koh--Stillman construct geometric syzygies out of special skew-symmetric matrices in \cite{koh-stillman}.\medskip
 
 For general canonical curves of even genus $g=2k$, these constructions yield syzygies of rank at most $k$. Green thus conjectured that $K_{k,1}(C,\omega_C)=0$, \cite{green-koszul}, which was proved by Voisin using a study of the cohomology of vector bundles on Hilbert schemes of a K3 surface, \cite{V1}. In order to replace this vanishing conjecture with a precise geometric statement, the following folklore conjecture has been around for quite some time:
 \begin{conjecture}[The Geometric Syzygy Conjecture in Even Genus] \label{geo-conj}
 For a general curve of genus $g=2k$, the last linear syzygy space $K_{k-1,1}(C,\omega_C)$ is spanned by geometric syzygies.
 \end{conjecture}

The conjecture above was studied by von Bothmer in his thesis from 2000, \cite{bothmer-thesis} and in the unpublished preprint \cite{bothmer-preprint}. Most significantly, von Bothmer proves Conjecture \ref{geo-conj} for curves of genus $g \leq 8$ in his 2007 paper \cite{bothmer-Transactions}.\medskip

On the other end of the spectrum, Green famously proved that, for a non-hyperelliptic curve $C$ of genus $g \geq 5$, the \emph{first} linear syzygy space $K_{1,1}(C,\omega_C)$ is generated by rank two syzygies of the form $\alpha \in K_{1,1}(C,\omega_C,H^0(L))$ for $L \in W^1_{g-1}(C)$, \cite{green-quadrics}. This was previously proven for a general curve by Andreotti--Mayer, \cite{andreotti-mayer}. For such a rank two syzygy $\alpha$, the variety $\rm{Syz}(\alpha)$ is a rank four quadric, defined by the Petri map $$H^0(C,L) \otimes H^0(C,\omega_C \otimes L^{-1}) \to H^0(C,\omega_C).$$ Identifying $K_{1,1}(C,\omega_C)$ with the space $(I_{C/\PP^{g-1}})_2$ of quadrics containing $C$, this upgrades Max Noether's classical theorem stating that $\dim (I_{C/\PP^{g-1}})_2=\binom{g-2}{2}$ into the finer statement that $(I_{C/\PP^{g-1}})_2$ is generated by quadrics of rank four. For additional background and motivation, see \cite{arbarello-harris}.  \medskip

Our main result is the following more precise version of Conjecture \ref{geo-conj}, providing an analogue of Green's result for the last linear syzygy group.
\begin{thm} \label{main-cor}
Let $C$ be a general canonical curve of genus $g=2k$. Then $K_{k-1,1}(C,\omega_C)$ is generated by the rank $k$ syzygies $$\alpha \in K_{k-1,1}(C,\omega_C,H^0(\omega_C \otimes A^{-1})), \; \; A \in W^1_{k+1}(C).$$
\end{thm}
\medskip
The above theorem extends Voisin's proof of Green's Conjecture \cite{V1} in precisely the same way that \cite{green-quadrics} extends Noether's Theorem.\footnote{Indeed, Green's Conjecture is equivalent to the statement that $K_{k-1,1}(C,\omega_C)$ has the lowest possible dimension $\binom{2k-1}{k-2}$. Also note that Noether's Theorem is equivalent to $K_{0,2}(C,\omega_C)=0$.} For any $\alpha \in K_{k-1,1}(C,\omega_C,H^0(\omega_C \otimes A^{-1}))$, the syzygy scheme $\rm{Syz}(\alpha)$ is the rational normal scroll $Y_A$ defined by the two-by-two minors of the matrix associated to the Petri map $H^0(A) \otimes H^0(\omega_C\otimes A^{-1}) \to H^0(\omega_C).$

%For curves of non-maximal gonality, a strengthened form of Conjecture \ref{geo-conj} is proven for the general such curve in \cite{lin-syz}, with extensions given in \cite{betti-multiple}, \cite{projecting}. Further, $K_{2,1}(C,\omega_C)$ is shown to be generated by geometric syzygies for general $C$ in \cite{bothmer-thesis}.
\medskip

 Theorem \ref{main-cor} is proven using K3 surfaces. Let $X$ be a K3 surface with $$\rm{Pic}(X)=\mathbb{Z}[L],\; \;(L)^2=4k-2.$$ Then $X$ carries a unique, stable, rank two bundle $E$ with invariants $\det(E)=L, \chi(E)=k+2$ and $c_2(E)=k+1$, \cite{lazarsfeld-BNP}. %For smooth $C \in |L|$, the dual bundle $E^{\vee}$ fits into the exact sequence
%$$0 \to E^{\vee} \to H^0(C,A) \otimes \mathcal{O}_X \to i_*A \to 0,$$
%where $A \in W^1_{k+1}(C)$ is a minimal pencil on $C$ and $i: C \hookrightarrow X$ is the inclusion.\medskip
For each section $s \in H^0(E)$, a construction due to Ein--Lazarsfeld produces a surjective map
 $$K_{k-1,2}(X,L) \twoheadrightarrow H^1(I_{Z(s)} \otimes L) \simeq \C,$$
 where $Z(s)$ is the zero locus of $s$, \cite{ein-lazarsfeld-asymptotic}. By Koszul duality, $K_{k-1,1}(X,L)^{\vee} \simeq K_{k-1,2}(X,L),$ \cite{green-koszul}. Thus dualizing the Ein--Lazarsfeld construction associates an element $$\alpha(s) \in \PP(K_{k-1,1}(X,L))$$
 to the section $s \in H^0(E)$, which we show is geometric of rank $k+1$. More precisely, $\alpha(s) \in K_{k-1,1}(X,L, H^0(L \otimes I_{Z(s)}))$.
 
 %We may now state our main technical result.
  \begin{thm} \label{thm-vero}
 The morphism 
 \begin{align*}
\PP(H^0(E)) & \to \PP(K_{k-1,1}(X,L))\\
 s & \mapsto \alpha(s)
 \end{align*} 
 is the Veronese embedding of degree $k-2$.\medskip
 
 In particular, there is a natural isomorphism $\rm{Sym}^{k-2} H^0(X,E) \simeq K_{k-1,1}(X,L).$
 \end{thm}
To avoid cluttering the introduction, we postpone giving an outline of the proof of Theorem \ref{thm-vero} to Section \ref{outline}. From Theorem \ref{thm-vero} we deduce:
 \begin{cor} \label{geo-K3}
 The Koszul group $K_{k-1,1}(X,L)$ is spanned by geometric syzygies.
 \end{cor}\medskip
 Theorem \ref{main-cor} then follows from Theorem \ref{thm-vero} using the Lefschetz Hyperplane Theorem and a result of Voisin.\medskip

We work exclusively over $\C$.\medskip

\textbf{Acknowledgements} The author is supported by NSF grant DMS-1701245. I thank Daniel Huybrechts and Robert Lazarsfeld for detailed comments on a draft of this paper. I thank Gavril Farkas for explaining the paper \cite{AFPRW} to me.

\section{Strategy and Comparison with Previous Approaches} \label{outline}
In this section, we outline our strategy and compare it to approaches from \cite{V1}, \cite{bothmer-preprint}, \cite{AFPRW}. \medskip

Von Bothmer's proposed approach \cite{bothmer-preprint} to Conjecture \ref{geo-conj} combines Voisin's Theorem \cite{V1} with a representation-theoretic study of the minimal free resolution of the Grassmannian in the Pl\"ucker embedding. Let $X$ be a K3 surface with $\rm{Pic}(X)=\mathbb{Z}[L]$ for $(L)^2=4k-2$, with Lazarsfeld--Mukai bundle $E$ as in the introduction. The starting point of \cite{bothmer-preprint} is the observation that \cite{V1} implies
 $$ \dim K_{k-1,1}(X,L) = \dim K_{k-1,1}(\rm{Gr}_2(H^0(E)^{\vee}), \mathcal{O}(1))= \binom{2k-1}{k-2},$$
 suggesting a link between the minimal free resolution of the K3 surface $X$ and the Grassmannian $\rm{Gr}_2(H^0(E)^{\vee})$.
 Further, there is a natural morphism
 $$X \hookrightarrow \PP(H^0(L)^{\vee}) \seq \PP(\wedge^2 H^0(E)^{\vee}),$$
 where $\PP(H^0(L)^{\vee}) \seq \PP(\wedge^2 H^0(E)^{\vee})$ arises from the map $\wedge^2 H^0(X,E) \twoheadrightarrow H^0(X,\det E)\simeq H^0(X,L).$
 This induces a map $$K_{k-1,1}(\rm{Gr}_2(H^0(E)^{\vee}), \mathcal{O}(1)) \to K_{k-1,1}(X,L) $$
 between vector spaces of the same dimension. \medskip
 
By representation-theoretic methods \cite{weyman-book} the last linear space $K_{k-1,1}(\rm{Gr}_2(H^0(E)^{\vee}), \mathcal{O}(1))$ of the Grassmannian is known to be generated by geometric syzygies. Conjecture \ref{geo-conj} then reduces to showing that the above square matrix of size $\binom{2k-1}{k-2}$ is nonsingular. \medskip

 Von Bothmer establishes the required nonsingularity of this matrix if $g \leq 8$ using Mukai's construction, \cite{bothmer-Transactions}. For large genus, however, this seems very difficult, since the intersection
 $$\rm{Gr}_2(H^0(E)^{\vee}) \cap \PP(H^0(L)^{\vee}) \seq \PP(\wedge^2 H^0(E)^{\vee})$$
 is far from transversal if $g>8$. \medskip
 
 Our approach replaces the role of the Grassmannian with the Ein--Lazarsfeld Secant construction of syzygies, \cite{ein-lazarsfeld-asymptotic}. As explained in the introduction, this produces a morphism
  \begin{align*}
\widetilde{\psi}\; : \;\PP(H^0(E))  \to \PP(K_{k-1,1}(X,L)) %\\
 %s & \mapsto \alpha(s)
 \end{align*} 
 corresponding to a surjection $\psi: K_{k-1,1}(X,L) \otimes \mathcal{O}_{\PP(H^0(E))} \twoheadrightarrow  \mathcal{O}_{\PP(H^0(E))}(k-2).$\medskip
 
Theorem \ref{thm-vero} then boils down to showing that the $\binom{2k-1}{k-2} \times \binom{2k-1}{k-2}$ matrix $H^0(\psi)$ has full rank. This is achieved by describing $\psi$ in terms of certain vector bundles on $\PP(H^0(E))$ arising as higher pushforwards from $B \to \PP(H^0(E))$, where $B$ is the blow-up of $X \times \PP(H^0(E))$ along the zero-locus $\mathcal{Z}$ of the universal section of $E \boxtimes \mathcal{O}(1)$. We then directly compute that the higher cohomology of the relevant bundles on $B$ vanish, which ultimately follows from the K\"unneth formula on $X \times \PP(H^0(E))$. \medskip
 
It is interesting to compare the isomorphism $\rm{Sym}^{k-2} H^0(X,E) \simeq K_{k-1,1}(X,L)$ with the Koszul module description of syzygies on the tangent developable, relating syzygy groups to (more complicated) $\rm{GL}(H^0(E))$ representations, \cite[\S 3]{eisenbud-orientation}, \cite{AFPRW}.  
 One should also compare with the isomorphism $\rm{Sym}^k H^0(X,E) \simeq \bigwedge^{k+1} H^0(X,L)$ discovered by Voisin, \cite{V1}. \medskip

We end this section with a few words on the relationship between the results of this paper and \cite{V1}. Theorem \ref{thm-vero} is stronger than Green's conjecture. Precisely, the natural isomorphism $\rm{Sym}^{k-2} H^0(X,E) \simeq K_{k-1,1}(X,L)$ implies in particular the count $\dim K_{k-1,1}(X,L)=\binom{2k-1}{k-2}$. This dimension count is equivalent to the vanishing $K_{k-2,2}(X,L)=0$ as conjectured by Green.\medskip

However, our proof, like von Bothmer's previous proposal \cite{bothmer-preprint}, relies on the result of \cite{V1} as input and so does not provide a new proof of Green's conjecture. We do, however, find a new approach to the problem. Indeed, a direct proof of injectivity of the map in Theorem \ref{real-thm} would be sufficient to provide a third proof of Green's conjecture in this case. We intend to pursue this approach in a future paper. \medskip

\section{Ein--Lazarsfeld's Secant Construction for K3 surfaces of even genus} \label{secant}
In this section, we apply the secant constructions from \cite[\S 2]{ein-lazarsfeld-asymptotic} to the case of K3 surfaces. Let $X$ be a K3 surface, $L$ an ample line bundle on $X$ with $(L)^2=2g-2$ for $g=2k\geq 4$ and $$\rm{Pic}(X) \simeq \mathbb{Z}[L].$$ A smooth $C \in |L|$ is a curve of genus $g$ and gonality $k+1$. Any line bundle $A \in W^1_{k+1}(C)$ induces a globally generated, rank two Lazarsfeld--Mukai bundle $E$, where the dual bundle $E^{\vee}$ fits into the short exact sequence
$$ 0 \to E^{\vee} \to H^0(C,A) \otimes \mathcal{O}_X \to i_*A \to 0,$$
for $i : C \hookrightarrow X$ the inclusion, \cite{lazarsfeld-BNP}. This bundle does not depend on either the choice of the curve $C$ or the line bundle $A$, and we have $\det E \simeq L$, $h^0(X,E)=k+2$. Pick a nonzero $s \in H^0(X,E)$ and let $Z:=Z(s)$ denote the zeroes of $s$. Then $Z$ is always zero-dimensional, \cite{V1}, and in fact corresponds to a $g^1_{k+1}$ on some $C' \in |L|$ ($Z$ may be a \emph{generalized} divisor, \cite{hartshorne-generalized}, when the integral curve $C'$ is singular).

We have a short exact sequence
$$0 \to \mathcal{O}_X \xrightarrow{s} E \to L \otimes I_Z \to 0.$$
We set $V=H^0(X,L)$ and define
$$W := V / H^0(L \otimes I_Z).$$
By the above short exact sequence, $W$ is a codimension $1=h^1(L\otimes I_Z)$ subset of $H^0(\mathcal{O}_Z(L))$, and, further $\dim W=k$. Define the torsion-free sheaf $\Sigma$ by the short exact sequence
$$0 \to \Sigma \to W\otimes \mathcal{O}_X \to \mathcal{O}_Z(L) \to 0,$$
where the morphism $W\otimes \mathcal{O}_X \to \mathcal{O}_Z(L)$ is the evaluation morphism. Define the kernel bundle $M_L$ by the exact sequence
$$0 \to M_L \to H^0(X,L) \otimes \mathcal{O}_X \to L \to 0.$$
There is a natural surjection $$M_L \twoheadrightarrow \Sigma$$ of sheaves. Recall that $K_{k-1,2}(X,L) \simeq H^1(X, \bigwedge^k M_L(L))$. \medskip

Following, \cite{ein-lazarsfeld-asymptotic}, one may construct syzygies in $K_{k-1,1}(X,L) \simeq {K_{k-1,2}(X,L)}^{\vee}$ by showing that the natural map
$$\phi \; : \; H^1(X, \wedge^k M_L(L)) \to H^1(X,\wedge^k \Sigma (L))$$ is surjective and, further, $H^1(X,\wedge^k \Sigma (L)) \neq 0$. \medskip

We first establish the latter:
\begin{lem} \label{lb-lemma}
We have $h^1(X,\wedge^k \Sigma (L))=1$.
\end{lem}
\begin{proof}
For any $x \in X$, 
$$\Sigma_x \simeq (I_{Z})_x \oplus \mathcal{O}_{X,x}^{\oplus k-1},$$
see \cite[\S 3]{ein-lazarsfeld-asymptotic}. Further, we have a natural surjection $\alpha: \wedge^k \Sigma \twoheadrightarrow I_Z$
which realizes $I_Z$ as the torsion-free quotient of $\wedge^k \Sigma$. The kernel of $\alpha$ has support on the zero-dimensional scheme $Z$. Thus $h^1(\wedge^k \Sigma(L))=h^1(I_Z(L))=1$.
\end{proof}
The kernel of the surjection $M_L \twoheadrightarrow \Sigma$ is a torsion-free sheaf $N$
fitting into the short exact sequence
$$0 \to N \to H^0(I_Z(L)) \otimes \mathcal{O}_X \xrightarrow{ev} I_Z(L) \to 0.$$
\begin{lem} \label{ses-N}
We have a short exact sequence
$$0 \to N \to H^0(E) \otimes \mathcal{O}_X \to E \to 0.$$ 
Furthermore, $N^{\vee}$ is the Lazarsfeld--Mukai bundle associated to $\omega_C \otimes A^{-1}$, i.e.\ we have a short exact sequence
$$ 0 \to N \to H^0(C,\omega_C \otimes A^{-1}) \otimes \mathcal{O}_X \to i_*(\omega_C \otimes A^{-1}) \to 0.$$
\end{lem}
\begin{proof}
The first claim follows immediately from the sequence $0 \to \mathcal{O}_X \xrightarrow{s} E \to L \otimes I_Z \to 0$ and the snake lemma. The second claim now follows similarly from the short exact sequence
$$0 \to H^0(A)^{\vee} \otimes \mathcal{O}_X \to E \to i_*(\omega_C \otimes A^{-1}) \to 0,$$
 see \cite{lazarsfeld-BNP}.
\end{proof}
In particular, $N$ is locally free and stable, \cite[\S 9.3]{huybrechts-k3}.
%\begin{prop} \label{N-stable}
%The vector bundle $N$ of rank $k$ is stable.
%\end{prop}
%\begin{proof}
%Suppose $F \seq N$ is a locally free subsheaf of rank $s<k$ with slope $\mu(F)>\mu(N)$. Then
%$$\det F \seq \wedge^s N \seq \wedge^s H^0(E) \otimes \mathcal{O}_X,$$
%showing that ${\det(F)}^{\vee}$ is effective, and hence $\det F=nL$ for $n \leq 0$. If $n \neq 0$ then
%$$\mu(F)=\frac{n(2g-2)}{s} < \frac{2-2g}{k}=\mu(N),$$
%which is a contradiction, so we must have $n=0$ and $\det F \simeq \mathcal{O}_X$. From
%$$H^0(E)^{\vee} \otimes \mathcal{O}_X \twoheadrightarrow N^{\vee} \twoheadrightarrow F^{\vee}$$
%we see that $F^{\vee}$ is globally generated, as is $F^{\vee}_{|_D}$ for a smooth divisor $D \sim mL$, $m \gg 0$. Observe that
%$H^0(F)=H^0(F_{|_D})=0$. But we have a short exact sequence
%$$0 \to F_{|_D} \to \mathcal{O}_s^{s+1} \to \det(F_{|_D}) \to 0,$$
%with $\det(F_{|_D}) \simeq \mathcal{O}_D$, \cite[Ch.\ 5]{huybrechts-lehn}. Thus we would have $h^0(\mathcal{O}_D) \geq s+1$, which is impossible.
%\end{proof}
The short exact sequence $0 \to N \to M_L \to \Sigma \to 0$ produces an exact sequence
$$N \otimes \wedge^{k-1}M_L \xrightarrow{g_2} \wedge^kM_L \xrightarrow{g_1} \wedge^k \Sigma \to 0,$$
\cite[Ch.\ 3, \S 7.2]{bourbaki-algebra}. 
%We can extend this to a (not necessarily exact)
%complex 
%$$F_{\bullet} \; : \; 0 \to S_k N \xrightarrow{g_{k+1}} S_{k-1}N \otimes M_L \ldots \to N \otimes \wedge^{k-1}M_L \xrightarrow{g_2} \wedge^kM_L \xrightarrow{g_1} \wedge^k \Sigma \to 0,$$
%see for instance \cite{weyman-sym-ext}. We record the following for later use.
%\begin{lem} \label{coh-codim}
%The complex $F_{\bullet}$ is exact outside of the zero-dimensional set $Z$.
%\end{lem}
%\begin{proof}
%Consider the localization $(F_{\bullet})_x$ at any $x \notin Z$. By \cite[Thm.\ 1]{weyman-sym-ext}, we need to show $I(f_{1,x})$ has depth at least $k$, for
%$f_{1,x}: N_x \to (M_L)_x$. Applying \cite[Prop.\ 20.8]{eisenbud-book}, this is equivalent to $\Sigma_x$ being locally free of rank $k$, which holds for $x \notin Z$.
%\end{proof}
\begin{prop} \label{phi-surj}
The natural map
$$\phi \; : \; H^1(X, \wedge^k M_L(L)) \to H^1(X,\wedge^k \Sigma (L))$$ is surjective.
\end{prop}
\begin{proof}
Set $K: \rm{Ker}( g_1: \wedge^kM_L \xrightarrow{g_1} \wedge^k \Sigma)$. It suffices to show $H^2(K(L))=0$. From the short exact sequence
$$0 \to \rm{Ker}(g_2) \to N \otimes \wedge^{k-1}M_L \to K \to 0,$$
and the fact that $\dim X=2$, it suffices to show $H^2(X, N \otimes \wedge^{k-1}M_L(L))=0$, which is Serre dual to $$H^0(X,\bigwedge^{k-1}M^{\vee}_L(-L) \otimes N^{\vee}) \simeq H^0(X,\bigwedge^{k+1}M_L \otimes N^{\vee}).$$
Since $\bigwedge^{k+1}M_L \otimes N^{\vee}$ has slope
$$(k+1)\mu(M_L)-\mu(N)=\frac{(k+1)(2-4k)}{2k}-\frac{2-4k}{k}<0,$$
it suffices to show that $\bigwedge^{k+1}M_L \otimes N^{\vee}$ is stable. But $M_L$ is stable, \cite[\S 9.3]{huybrechts-k3}, as is $N$, so this follows from \cite[Ch.\ 3]{huybrechts-lehn}.
\end{proof}

\section{Relativizing the secant construction}
We now relative the above construction. We work on $\PP(H^0(E)):=\rm{Proj}(H^0(E)^{\vee})$, with our conventions such that $H^0(\mathcal{O}_{\PP(H^0(E))}(1))=H^0(X,E)^{\vee}$.
Let $p: X \times \PP(H^0(E)) \to X$ and $q: X \times \PP(H^0(E)) \to \PP(H^0(E))$ denote the projections. On $X \times \PP(H^0(E))$ we consider the rank two bundle $E \boxtimes \mathcal{O}(1)$. We have 
$$H^0(E \boxtimes \mathcal{O}(1))\simeq H^0(E)\otimes H^0(E)^{\vee}\simeq \rm{Hom}(H^0(E),H^0(E)).$$
Define $\mathcal{L}:=\det(E \boxtimes \mathcal{O}(1))\simeq L \boxtimes \mathcal{O}(2)$.
Thus the identity $id \in \rm{Hom}(H^0(E),H^0(E))$ induces a natural short exact sequence
$$0 \to \mathcal{O}_{X \times \PP(H^0(E))} \to E \boxtimes \mathcal{O}(1) \to I_{\mathcal{Z}} \otimes \mathcal{L} \to 0,$$
where $\mathcal{Z}$ is defined to be the zero locus of $id \in H^0(E \boxtimes \mathcal{O}(1))$. Twist the above sequence by $q^*\mathcal{O}_{\PP(H^0(E))}(-2)$ to obtain
$$0 \to q^*\mathcal{O}(-2) \to E\boxtimes \mathcal{O}(-1) \to I_{\mathcal{Z}} \otimes p^*L \to 0.$$
Note that $q_*( I_{\mathcal{Z}} \otimes p^*L)$ is a vector bundle of rank $k+1$ on $\PP(H^0(E))$.\medskip

\begin{lem} \label{deg-calc}
We have $\deg q_*( I_{\mathcal{Z}} \otimes p^*L)=-k$ and $\deg R^1 q_*( I_{\mathcal{Z}} \otimes p^*L)=-2 $.
\end{lem}
\begin{proof}
Applying $q_*$ we get the short exact sequence 
$$0 \to \mathcal{O}_{\PP(H^0(E))}(-2) \to H^0(E) \otimes \mathcal{O}_{\PP(H^0(E))}(-1) \to q_*( I_{\mathcal{Z}} \otimes p^*L) \to 0,$$
and the first claim follows. The second claim follows from the isomorphism $$R^1q_*( I_{\mathcal{Z}} \otimes p^*L) \simeq R^2q_*\mathcal{O}_{X \times \PP(H^0(E))} \otimes \mathcal{O}_{ \PP(H^0(E))}(-2),$$
combined with the fact that $R^2q_*\mathcal{O}_{X \times \PP(H^0(E))}  \simeq \mathcal{O}_{ \PP(H^0(E))}$ by relative duality.
\end{proof}
There is a natural, surjective morphism $$q^*q_*(p^*L) \twoheadrightarrow p^*L,$$
which restricts on fibres of $q$ to the evaluation map $H^0(L) \otimes \mathcal{O}_X \twoheadrightarrow L$. We define $$\mathcal{M}:=\rm{Ker}(q^*q_*(p^*L) \twoheadrightarrow p^*L).$$
Observe that $\mathcal{M} \simeq p^* M_L.$ Next, define
$$\mathcal{W}:=\rm{Coker}(q_*( I_{\mathcal{Z}}\otimes p^*L) \hookrightarrow q_*p^*L),$$
which is a rank $k$ vector bundle on $\PP(H^0(E))$ fitting into a short exact sequence
$$0 \to \mathcal{W} \to q_*( {p^*L}_{|_{\mathcal{Z}}}) \to R^1q_*(p^*L\otimes I_{\mathcal{Z}}) \to 0.$$
We have $\det(\mathcal{W})\simeq \mathcal{O}_{\PP(H^0(E))}(k)$ by Lemma \ref{deg-calc}. \medskip

Applying $q^*$, we have a morphism $q^*\mathcal{W} \to q^*q_*( {p^*L}_{|_{\mathcal{Z}}})$ which we compose with the natural map $q^*q_*( {p^*L}_{|_{\mathcal{Z}}}) \to p^*L_{|_{\mathcal{Z}}}$ to obtain a surjective morphism 
$$q^*\mathcal{W} \twoheadrightarrow p^*L_{|_{\mathcal{Z}}},$$
coinciding with $W \otimes \mathcal{O}_X \twoheadrightarrow \mathcal{O}_Z(L)$ on fibres of $q$. Define
$$\mathbf{\Sigma}:=\rm{Ker}(q^*\mathcal{W} \twoheadrightarrow p^*L_{|_{\mathcal{Z}}}).$$
\begin{lem}
We have a natural surjection 
$$\wedge^k \mathbf{\Sigma} \to I_{\mathcal{Z}} \otimes q^*\mathcal{O}(k).$$
\end{lem}
\begin{proof}
We have a natural map $\wedge^k \mathbf{\Sigma}  \to \wedge^k q^* \mathcal{W} \simeq q^*\mathcal{O}(k)$. The remainder of the proof follows as in \cite[Corollary 3.7]{ein-lazarsfeld-asymptotic}.
\end{proof}

We have a commutative diagram
$$\begin{tikzcd}
0 \arrow[r] & \mathcal{M} \arrow[r] \arrow[d] & q^*q_*(p^*L)  \arrow[r] \arrow[d] &p^*L \arrow[r] \arrow[d] &0\\
0 \arrow[r] & \mathbf{\Sigma} \arrow[r]  & q^*\mathcal{W}  \arrow[r]  &  p^*L_{|_{\mathcal{Z}}} \arrow[r] & 0
\end{tikzcd}$$
We can now globalize the map $\phi \; : \; H^1(X, \wedge^k M_L(L)) \to H^1(X,\wedge^k \Sigma (L))$, whose definition depended on a choice of $s \in H^0(E)$, to the natural morphism
$$\psi : H^1(X,\wedge^kM_L (L)) \otimes \mathcal{O}_{\PP(H^0(E))} \simeq R^1q_*(\bigwedge^k \mathcal{M} (p^*L)) \twoheadrightarrow R^1q_*(\bigwedge^k \mathbf{\Sigma} (p^*L)) .$$ Note that this morphism is indeed surjective, by Nakayama's Lemma and Proposition \ref{phi-surj}. 
\begin{lem} The sheaf $R^1q_*(\bigwedge^k \mathbf{\Sigma} (p^*L)) $ is a line bundle on $\PP(H^0(E))$ of degree $k-2$.
\end{lem}
\begin{proof}
The sheaf $R^1q_*(\bigwedge^k \mathbf{\Sigma} (p^*L)) $ is a line bundle by Lemma \ref{lb-lemma}. Further, the canonical surjection, $\bigwedge^k \mathbf{\Sigma} \to I_{\mathcal{Z}}\otimes q^*\mathcal{O}(k)$ induces an isomorphism $R^1q_* (\bigwedge^k \mathbf{\Sigma} (p^*L)) \simeq R^1q_*(I_{\mathcal{Z}}(p^*L))\otimes \mathcal{O}(k)$. The claim now follows from Lemma \ref{deg-calc}.
\end{proof}
Thus $\psi$ provides a surjection
$$H^1(X,\wedge^kM_L (L)) \otimes \mathcal{O}_{\PP(H^0(E))} \twoheadrightarrow \mathcal{O}_{\PP(H^0(E))}(k-2).$$
Recalling that $H^1(X,\wedge^kM_L (L)) \simeq K_{k-1,2}(X,L) \simeq K_{k-1,1}(X,L)^{\vee}$, we obtain:
\begin{prop}
We have a natural morphism 
$$\widetilde{\psi} \; : \; \PP(H^0(E)) \to \PP(K_{k-1,1}(X,L))$$
of degree $k-2$.
\end{prop}
On the other hand, note that, by Voisin's Theorem \cite{V1} we know
$$\dim K_{k-1,1}(X,L) = \binom{2k-1}{k-2}=\dim H^0(\mathcal{O}_{\PP(H^0(E))}(k-2)),$$
see \cite[\S 4.1]{farkas-progress}. Thus, one expects that $\widetilde{\psi}$ is the Veronese embedding of degree $k-2$. To show that this is actually we case, it suffices to show:
\begin{thm} \label{main-thm}
The morphism 
$$H^0(\psi) \; : \; H^1(X,\wedge^kM_L (L)) \to H^0(R^1q_*(\bigwedge^k \mathbf{\Sigma} (p^*L)))$$
is surjective.
\end{thm}
We will prove the theorem above in the next section. 
%We end this section with a definition. Let $\mathcal{N}$ be defined by the exact sequence
%$$0 \to \mathcal{N} \to q^*q_* I_{\mathcal{Z}}(p^*L) \to I_{\mathcal{Z}}(p^*L) \to 0.$$
%Then we have a short exact sequence
%$$0 \to \mathcal{N} \to \mathcal{M} \to \mathbf{\Sigma} \to 0.$$
%Arguing as in Lemma \ref{ses-N} using the sequence
%$$0 \to q^*\mathcal{O}(-2) \to E\boxtimes \mathcal{O}(-1) \to I_{\mathcal{Z}} \otimes p^*L \to 0,$$
%we have 
%$$0 \to \mathcal{N} \to q^*q_*E \otimes q^*\mathcal{O}(-1) \to E \otimes q^*\mathcal{O}(-1) \to 0,$$
%showing that $\mathcal{N}$ is a vector bundle of rank $k$. Further, by comparison with Lemma \ref{ses-N}, this shows $\mathcal{N} \simeq N \boxtimes \mathcal{O}(-1).$

\section{Surjectivity of $H^0(\psi)$}
To ease the notation, set $\PP:=\PP(H^0(E))$. Let $\pi: B \to X \times \PP$ denote the blow-up at $\mathcal{Z} \seq X \times \PP$ and let $D$ denote the exceptional divisor. Note that $\mathcal{Z}$ is a local complete intersection in $X \times \PP$, since it is a codimension two locus defined by the zeroes of a section of a rank two vector bundle. By \cite[\S 5]{kovacs-rational}, we have that $D \to \mathcal{Z}$ is a projective bundle, and further $R\pi_* \mathcal{O}_D \simeq \mathcal{O}_Z$, $R \pi_* \mathcal{O}_B \simeq \mathcal{O}_{X \times \PP}$ under the natural morphisms. Further, applying $\pi_*$ to 
$$0 \to I_D \to \mathcal{O}_B \to \mathcal{O}_D \to 0,$$
we see $\pi_*I_D \simeq I_{\mathcal{Z}}$ and $R^i\pi_*I_D=0, i>0$.\medskip

We now adapt the Ein--Lazarsfeld Secant construction to this setting. Let $p': B \to X$, $q': B \to \PP$ be defined by $p'=p \circ \pi, q'=q \circ \pi$. By the projection formula and the above we have canonical identifications
$$q'_* ({p'}^* L \otimes I_D) \simeq q_*(p^*L \otimes I_{\mathcal{Z}}), \;\; q'_*{p'}^*L \simeq q_*p^*L.$$
This lets us write the rank $k$ vector bundle $\mathcal{W}$ as
$$\mathcal{W}= \rm{Coker}\left(q'_* ({p'}^* L \otimes I_D) \to q'_*{p'}^*L\right).$$
Further, the natural composition
$${q'}^* \mathcal{W} \to {q'}^*q'_*({p'}^*L_{|_D}) \to {p'}^* L_{|_D},$$
is just $\pi^*$ applied to the natural, surjective map $q^* \mathcal{W} \to p^* L_{|_\mathcal{Z}}$. Hence ${q'}^* \mathcal{W} \twoheadrightarrow {p'}^* L_{|_D}$ is surjective and we define
$$\Gamma:= \rm{Ker}\left({q'}^* \mathcal{W} \twoheadrightarrow {p'}^* L_{|_D} \right),$$
which will replace $ \mathbf{\Sigma}$ in the setting of the blow-up.\medskip

The advantage of passing to the blow-up is revealed in the following:
\begin{lem}
The sheaf $\Gamma$ is locally free of rank $k$. Furthermore, $\wedge^k \Gamma \simeq I_D \otimes {q'}^* \mathcal{O}(k)$.
\end{lem}
\begin{proof}
It suffices to show that, for any $x \in B$ and any $\mathcal{O}_{B,x}$ module $N$, $\rm{Ext}^1_{\mathcal{O}_{B,x}}(\Gamma_x,N)=0$. From $0 \to \Gamma \to q'^* \mathcal{W} \to {p'}^* L_{|_D} \to 0,$ it, in turn, suffices to show $\rm{Ext}^2_{\mathcal{O}_{B,x}}(\mathcal{O}_{D,x},N)=0$, which follows from the fact that $I_D$ is locally free.\medskip

The claim $\wedge^k \Gamma \simeq I_D \otimes {q'}^* \mathcal{O}(k)$ follows immediately by taking determinants of the exact sequence $0 \to \Gamma \to {q'}^* \mathcal{W} \to {p'}^* L_{|_D} \to 0.$
\end{proof}

We have a surjection
$${q'}^*q'_*({p'}^*L \otimes I_D) \twoheadrightarrow {p'}^*L \otimes I_D$$
obtained by composing
${q'}^*q'_*({p'}^*L \otimes I_D) \twoheadrightarrow {p'}^*L \otimes \pi^*I_{\mathcal{Z}}$ with the natural surjection ${p'}^*L \otimes \pi^*I_{\mathcal{Z}} \to {p'}^*L \otimes I_D$.
Define the vector bundle $\mathcal{S}$ by the short exact sequence
$$0 \to \mathcal{S} \to {q'}^*q'_*({p'}^*L \otimes I_D) \to {p'}^*L \otimes I_D \to 0.$$ We have an exact sequence of vector bundles
$$0 \to \mathcal{S} \to \pi^* \mathcal{M} \to \Gamma \to 0.$$
\begin{thm} \label{real-thm}
The natural morphism $H^1(B, \wedge^k \pi^* \mathcal{M} \otimes {p'}^*L) \to H^1(B, \wedge^k \Gamma \otimes {p'}^*L)$ is surjective.
\end{thm}
\begin{proof}
We have an exact sequence
\begin{align*}
0 &\to \rm{Sym}^k \mathcal{S} \otimes {p'}^*L \to \rm{Sym}^{k-1} \mathcal{S} \otimes \pi^*\mathcal{M}\otimes {p'}^*L \to \ldots \\
 & \to \mathcal{S} \otimes \wedge^{k-1} \pi^* \mathcal{M}   \otimes {p'}^*L \to \wedge^k  \mathcal{M}   \otimes {p'}^*L \to \wedge^k \Gamma \otimes {p'}^*L \to 0,
\end{align*}
\cite{weyman-sym-ext}. It thus suffices to show $$H^{1+i}(\rm{Sym}^i \mathcal{S} \otimes \wedge^{k-i} \pi^* \mathcal{M} \otimes {p'}^*L)=0$$
for $1 \leq i \leq k$. Next, from the short exact sequence
$$0 \to \rm{Sym}^{i} \mathcal{S} \to \rm{Sym}^i\left({q'}^*q'_*({p'}^*L \otimes I_D)\right) \to \rm{Sym}^{i-1}\left({q'}^*q'_*({p'}^*L \otimes I_D)\right)\otimes {p'}^*L \otimes I_D \to 0$$
it suffices to show 
\begin{align*}
&H^{1+i}\left(B,\rm{Sym}^i\left({q'}^*q'_*({p'}^*L \otimes I_D)\right) \otimes \wedge^{k-i} \pi^* \mathcal{M} \otimes {p'}^*L\right)=0,\\
&H^i\left(B,\rm{Sym}^{i-1}\left({q'}^*q'_*({p'}^*L \otimes I_D)\right)\otimes \wedge^{k-i} \pi^* \mathcal{M} \otimes {p'}^*L^{\otimes 2} \otimes I_D\right)=0
\end{align*}
By the Leray spectral sequence, this is equivalent to 
\begin{align*}
&H^{1+i}\left(X \times \PP,\rm{Sym}^i\left({q}^*q_*({p}^*L \otimes I_{\mathcal{Z}})\right) \otimes \wedge^{k-i} \mathcal{M} \otimes {p}^*L\right)=0,\\
&H^i\left(X \times \PP,\rm{Sym}^{i-1}\left({q}^*q_*({p}^*L \otimes I_{\mathcal{Z}})\right)\otimes \wedge^{k-i}  \mathcal{M} \otimes {p}^*L^{\otimes 2} \otimes I_{\mathcal{Z}}\right)=0
\end{align*}
By taking $\rm{Sym}^i$ of the exact sequence
$$0 \to q^*\mathcal{O}(-2) \to q^*q_*p^*E \otimes q^*\mathcal{O}(-1) \to q^*q_*(p^*L \otimes I_{\mathcal{Z}})\to 0$$
it suffices to show
\begin{align*}
&H^{2+i}\left(X \times \PP, \wedge^{k-i} M_L (L) \boxtimes \rm{Sym}^{i-1} \left(q_*p^*E \right)(-i-1)  \right)=0,\\
&H^{1+i}\left(X \times \PP, \wedge^{k-i} M_L (L) \boxtimes \rm{Sym}^i \left(q_*p^*E \right)(-i)  \right)=0,\\
&H^{1+i}\left(X \times \PP, \left(\wedge^{k-i} M_L (2L) \boxtimes \rm{Sym}^{i-2} \left(q_*p^*E \right)(-i)  \right)\otimes I_{\mathcal{Z}} \right)=0,\\
&H^{i}\left(X \times \PP, \left(\wedge^{k-i} M_L (2L) \boxtimes \rm{Sym}^{i-1} \left(q_*p^*E \right)(-i+1)  \right)\otimes I_{\mathcal{Z}} \right)=0.
\end{align*}
Since $q_*p^*E \simeq H^0(X,E) \otimes \mathcal{O}_{\PP}$ is a trivial bundle on $\PP \simeq \PP^{k+1}$, the first two claims follow from the K\"unneth formula, for $1 \leq i \leq k$. For the last two claims, we use the short exact sequence
$$0 \to L^{-1} \boxtimes \mathcal{O}(-2) \to E(L^{-1}) \boxtimes \mathcal{O}(-1) \to I_{\mathcal{Z}} \to 0,$$
so it suffices to have the four vanishings
\begin{align*}
&H^{2+i}\left(\wedge^{k-i} M_L (L) \boxtimes \rm{Sym}^{i-2} \left(q_*p^*E \right)(-i-2)  \right)=0,\\
&H^{1+i}\left( \wedge^{k-i} M_L \otimes E(L) \boxtimes \rm{Sym}^{i-2} \left(q_*p^*E \right)(-i-1)  \right)=0,\\
&H^{1+i}\left( \wedge^{k-i} M_L (L) \boxtimes \rm{Sym}^{i-1} \left(q_*p^*E \right)(-i-1)  \right)=0,\\
&H^{i}\left(\wedge^{k-i} M_L \otimes E(L) \boxtimes \rm{Sym}^{i-1} \left(q_*p^*E \right)(-i)  \right)=0.
\end{align*}
This follows from the K\"unneth formula (using $H^1(X,L)=0$ if $i=k$ in the first vanishing).
\end{proof}
We now immediately deduce the proof of Theorem \ref{main-thm}.
\begin{proof}[Proof of Theorem \ref{main-thm}]
Recall that $R^1q_*(\bigwedge^k \mathbf{\Sigma} (p^*L)) \simeq R^1q_*\left( (L \boxtimes \mathcal{O}(k)) \otimes I_{\mathcal{Z}} \right) $. We have natural isomorphisms
\begin{align*}
H^1(X \times \PP, \wedge^k \mathcal{M} (p^*L)) &\simeq H^1(B, \wedge^k \pi^* \mathcal{M} \otimes {p'}^*L) \\
H^1(X \times \PP, (L \boxtimes \mathcal{O}(k)) \otimes I_{\mathcal{Z}}) &\simeq H^1(B, \pi^*(L \boxtimes \mathcal{O}(k)) \otimes I_D) \simeq H^1(B, \wedge^k \Gamma \otimes {p'}^*L)
\end{align*}
so that the natural map $H^1(X \times \PP, \wedge^k \mathcal{M} (p^*L)) \to H^1(X \times \PP, (L \boxtimes \mathcal{O}(k)) \otimes I_{\mathcal{Z}})$ is surjective by Theorem \ref{real-thm}.
Next, $q_*\wedge^k \mathcal{M} (p^*L) \simeq H^0(X,\wedge^k M_L(L)) \otimes \mathcal{O}_{\PP}$ so that $H^i(q_*\wedge^k \mathcal{M} (p^*L))=0$ for $i>0$. Furthermore, the short exact sequence
$$0 \to \mathcal{O}_{\PP}(k-2) \to q_*E \otimes \mathcal{O}_{\PP}(k-1) \to q_*\left((L \boxtimes \mathcal{O}(k)) \otimes I_{\mathcal{Z}}\right)\to 0$$
shows $H^i(q_*\left((L \boxtimes \mathcal{O}(k)) \otimes I_{\mathcal{Z}}\right))=0$ for $i>0$. Hence the Leray spectral sequence produces natural isomorphisms
\begin{align*}
H^1(X \times \PP, \wedge^k \mathcal{M} (p^*L))  &\simeq  H^0(R^1q_*\wedge^k \mathcal{M} (p^*L))\\
H^1(X \times \PP, (L \boxtimes \mathcal{O}(k))\otimes I_{\mathcal{Z}})&\simeq H^0(R^1q_*(L \boxtimes \mathcal{O}(k)) \otimes I_{\mathcal{Z}}), 
\end{align*}
giving the claim.
\end{proof}
As a corollary, we see that $\widetilde{\psi}$ is the Veronese embedding of degree $k-2$, and that there is a natural isomorphism $K_{k-1,1}(X,L) \simeq \rm{Sym}^{k-2}\; H^0(X,E)$.
\begin{thm} \label{geometric-syz}
Let $X$ be a K3 surface of even genus $g=2k$ with $\rm{Pic}(X) \simeq \mathbb{Z}[L]$. The syzygy space $K_{k-1,1}(X,L)$ is spanned by geometric syzygies of rank at most $k+1$.
\end{thm}
\begin{proof}
We follow the notation from section \ref{secant}. It remains to show that, for any fixed nonzero $s \in H^0(X,E)$ and any nonzero $\alpha \in H^1(X, \wedge^k \Sigma(L))^{\vee}$, then
$$\phi^{\vee}(\alpha) \in H^1(X,\wedge^kM_L(L))^{\vee} \simeq K_{k-1,1}(X,L)$$
is a syzygy of rank at most $k+1$. Let $\pi: \widetilde{X} \to X$ denote the blow-up of $X$ at $Z:=Z(s)$, with exceptional divisor $D$. We may naturally identify $W=H^0(X,L)/H^0(X,L \otimes I_Z)$ with $H^0(\widetilde{X},\pi^*L)/H^0(\widetilde{X},\pi^*L \otimes I_D)$. Define
$$ \gamma:= \rm{Ker}(W \otimes \mathcal{O}_{\widetilde{X}} \twoheadrightarrow \pi^*L_{|_D}).$$
There is a short exact sequence
$$0 \to S \to \pi^*M_L \to \gamma \to 0,$$
where $S$ is the kernel of the surjection $H^0(\pi^*L \otimes I_D) \otimes \mathcal{O}_{\widetilde{X}} \to \pi^*L \otimes I_D$. 
Both $S$ and $\gamma$ are vector bundles, of the same rank $k$. Further, $\phi$ may be identified with the natural map 
$$H^1(\pi^*\wedge^kM_L(L)) \to H^1(\wedge^k\gamma (L)) \simeq H^1(I_D(\pi^*L)).$$
Note that $\wedge^k M_L (L) \simeq \wedge^k M_L^{\vee}$ and $I_D(\pi^*L) \simeq \wedge^k S^{\vee}$. We may thus identify $\phi$ with a natural, surjective map
$$H^1(\pi^* \wedge^k M^{\vee}_L) \to H^1(\wedge^k S^{\vee})$$ arising from the surjection $\pi^*M^{\vee}_L \twoheadrightarrow S^{\vee}$. This map is Serre dual to an inclusion
$H^1(\wedge^k S (D)) \to H^1(\wedge^k M_L(D))$ which we may view as an inclusion of Koszul spaces
$$K_{k-1,1}(\widetilde{X},D,\pi^*L-D) \hookrightarrow K_{k-1,1}(\widetilde{X},D,\pi^*L).$$ Since
$H^0(\widetilde{X}, \pi^* L^{\otimes n}(D)) \simeq H^0(\widetilde{X}, \pi^* L^{\otimes n}) $ for any $n$, the natural map $$K_{k-1,1}(\widetilde{X},\pi^*L) \to K_{k-1,1}(\widetilde{X},D,\pi^*L)$$ is an isomorphism, and further $K_{k-1,1}(\widetilde{X},\pi^*L) \simeq K_{k-1,1}(X,L)$. This map identifies the subspace
$$K_{k-1,1}(\widetilde{X},\pi^*L, H^0(\pi^*L-D)) \seq K_{k-1,1}(\widetilde{X},\pi^*L)$$ with $K_{k-1,1}(\widetilde{X},D,\pi^*L-D)$. Thus syzygies in the image of $\phi^{\vee}$ have rank at most $h^0(\pi^*L-D)=k+1$.

\end{proof}
The argument in the following Corollary is essentially due to Voisin, \cite[Proof of Prop.\ 7]{V1} and also appears in \cite{bothmer-preprint}.
\begin{cor}
Let $C$ be a general canonical curve of genus $g=2k$. Then $K_{k-1,1}(C,\omega_C)$ is generated by syzygies of rank $k$. More precisely, $K_{k-1,1}(C,\omega_C)$ is generated by the subspaces $K_{k-1,1}(C,\omega_C, H^0(\omega_C \otimes A^{\vee}))$, where $A \in W^1_{k+1}(C)$ is a minimal pencil.
\end{cor}
\begin{proof}
Let $(X,L)$ be as in Theorem \ref{geometric-syz}, and let $C \in |L|$ be general. We have a surjective, finite morphism
$$d: \rm{Gr}_2(H^0(E)) \to |L|$$
which takes a two dimensional subspace $W \seq H^0(E)$ to the degeneracy locus of the evaluation morphism $W \otimes \mathcal{O}_X \xrightarrow{ev} E$, \cite{V1}.\medskip

 We may naturally identify the cokernel of $ev$ with $K_C \otimes A$, where $C=d(W) \in |L|$ and $A \in W^1_{k+1}(C)$, \cite{lazarsfeld-BNP}, \cite{aprodu-farkas}. If $C$ is sufficiently general, then $d^{-1}(C) \simeq W^1_{k+1}(C)$ is reduced, where $W^1_{k+1}(C)$ is the Brill--Noether locus of minimal pencils on $C$. Under the identification $d^{-1}(C) \simeq W^1_{k+1}(C)$, a line bundle $A \in W^1_{k+1}(C)$ is mapped to $H^0(A)^{\vee}\seq H^0(E)$. \medskip

 By \cite[Proof of Prop.\ 7]{V1}, the spaces $\rm{Sym}^{k-2}H^0(A)^{\vee}$, $A \in W^1_{k+1}(C)$ generate $\rm{Sym}^{k-2}H^0(E)$. More geometrically, each such $H^0(A)^{\vee}$ corresponds to a line $T_i \seq \PP(H^0(E))$. Set $$T:=\bigcup_{A \in W^1_{k+1}(C)} T_i$$ to be the union of these lines. Thus the image of $T$ under $$\widetilde{\psi}: \PP(H^0(E)) \to \PP(K_{k-1,1}(X,L))$$ is non-degenerate. Hence $K_{k-1,1}(X,L)$ is generated by the subspaces $K_{k-1,1}(X,L, H^0(L \otimes I_s))$, where $s \in H^0(E)$ is a section corresponding to a $g^1_{k+1}$ on the fixed curve $C$. But such spaces are identified with subspaces of the form $$K_{k-1,1}(C,\omega_C, H^0(\omega_C \otimes A^{\vee})) \seq K_{k-1,1}(C,\omega_C),\; \; A \in W^1_{k+1}(C)$$ under the isomorphism $K_{k-1,1}(X,L) \simeq K_{k_1,1}(C, \omega_C)$, \cite{green-koszul}.
\end{proof}

\end{document}